\documentclass{birkjour}
\usepackage{enumerate}

\newtheorem{thm}{Theorem}[section]

\newtheorem{lemma}[thm]{Lemma}
\newtheorem{prop}[thm]{Proposition}
\theoremstyle{definition}
\newtheorem{defn}[thm]{Definition}
\theoremstyle{remark}

\newtheorem{ex}{Example}
\numberwithin{equation}{section}

\newcommand{\eps}{\varepsilon}
\newcommand{\Real}{\mathbb R}\newcommand{\Complex}{\mathbb C}

\newcommand{\B}{\mathcal{B}}

\newcommand{\h}{\mathcal{H}}

\newcommand{\N}{\mathcal{N}}
\newcommand{\M}{\mathcal{M}}
\newcommand{\R}{\mathcal{R}}

\newcommand{\V}{\mathcal{V}}
\newcommand{\W}{\mathcal{W}}

\newcommand{\rlin}[1]{\mathcal{B}(#1)}

\DeclareMathOperator*{\spn}{span}
\DeclareMathOperator*{\re}{Re}

\begin{document}

\title[Weyl--von Neumann Theorem for Antilinear Operators]{On a Weyl--von Neumann -type Theorem for Antilinear Self-adjoint Operators}
\author[S. Ruotsalainen]{Santtu Ruotsalainen}
\address{
Aalto University \\
Institute of Mathematics \\
P.O. Box 11100 \\
FI-00076 Aalto \\
Finland}
\email{Santtu.Ruotsalainen [at] aalto.fi}

\date{\today}
\begin{abstract}
Antilinear operators on a complex Hilbert space arise in various contexts in mathematical physics. In this paper, an analogue of  the Weyl--von Neumann theorem for antilinear self-adjoint operators is proved, i.e. that an antilinear self-adjoint operator is the sum of a diagonalizable operator and of a compact operator with arbitrarily small Schatten $p$-norm. In doing so, we discuss conjugations and their properties. A spectral integral representation for antilinear self-adjoint operators is constructed. 
\end{abstract}

\maketitle
\tableofcontents

\section{Introduction}

This paper is concerned with an analogue of the Weyl--von Neumann theorem for self-adjoint antilinear operators on a complex Hilbert space $\h$. The Weyl--von Neumann--Berg theorem states that a complex linear normal operator is the sum of a diagonalizable and an arbitrarily small compact operator. An operator is diagonalizable if  it has an orthonormal set of eigenvectors spanning $\h$. Motivated by the occurrence of antilinear operators in a wide range of mathematical physics applications, it is natural to inquire in what respect there is a Weyl--von Neumann -type theorem for antilinear operators. It is shown that any antilinear self-adjoint operator is the sum of a diagonalizable operator and an operator of arbitrarily small Schatten $p$-norm. 

The Weyl--von Neumann theorem and its ramifications have been near the center of operator theory for the majority of the past century and have led to new operator theoretical techniques. In 1909 Weyl proved that a self-adjoint complex linear operator can be diagonalized modulo an arbitrarily small compact operator \cite{Weyl1909}. In 1935, von Neumann extended the result to unbounded operators and showed that the compact operator can be taken to be Hilbert-Schmidt \cite{vonN1935}. Halmos \cite{Halmos1970} raised the question if there exists an analogous result for normal operators. Berg \cite{Berg1971} and Halmos himself \cite{Halmos1972}, with a different technique, proved that any normal operator is diagonalizable modulo a compact operator. Voiculescu's result is the most general one stating that also for normal operators  the compact perturbation has an  arbitrarily small Hilbert--Schmidt norm \cite{Voiculescu1976}. 

For complex linear operators on finite dimensional spaces, unitary diagonalizability is one of the many equivalent conditions for  normality \cite{grone.etal:normals, elsner.ikramov:normals}. For antilinear operators, and for real linear operators in general, these conditions are no longer equivalent and the notion of normality is not so straightforward. Antilinear self-adjoint operators on finite dimensional spaces can be unitarily diagonalized, whereas an antilinear operator commuting with its adjoint may not allow for such in general. This is how self-adjointness of antilinear operators is an analogue of normality of complex linear operators. Therefore, it is a natural question to ask what is the analogue of the Weyl--von Neumann theorem in the antilinear infinite dimensional setting. In addition, as the spectral theory of real linear operators in general is not totally understood, a Weyl--von Neumann -type theorem would be informative about spectral properties of antilinear self-adjoint operators. 

Antilinear operators appear in a plethora of applications and their usefulness has not remained unnoticed \cite{herbut:basicalgebra, DidenkoSilbermann, reallinearoperator}. Antilinear operators are found in the study of planar elasticity  in the form of the Friedrichs operator \cite{putinarshapiro}. In inverse problems they arise in solving the so-called $\partial$-bar equation in the plane \cite{astalapaivarinta}. In quantum mechanics antilinearity is classically noted in studying time reversal but antilinear operators are useful in the Hartree--Bogolyubov theory in nuclear physics \cite{herbut:basicalgebra}, in studying quantum entanglement \cite{AV:EPR} and quantum teleportation \cite{KKAJQTelep}. In addition, self-adjoint antilinear operators are naturally linked to complex symmetric operators which are of importance in mathematical physics \cite{garcia, garcia2, garcia:newclasses}. 

The paper is organized as follows. In Section 2, notation, basic notions and properties of real linear and antilinear operators are given. Conjugations and their diagonalizability are studied and antilinear projections onto a closed subspace are discussed. In Section 3, the representation of an antilinear self-adjoint operator as a spectral integral is developed. In Section 4, the main theorem, Weyl--von Neumann theorem for antilinear self-adjoint operators, is proved. The connection to complex symmetric operators is presented. 

\section{Antilinear operators and conjugations}

Let $\h$ be a separable Hilbert space over $\Complex$. An operator $A$ on $\h$ is said to be real linear, if it is additive and commutes with real numbers. It is called complex linear if in addition it satisfies $Ai=iA$, or antilinear if it satisfies $Ai=-iA$. The set of real linear operators is a real Banach algebra with the operator norm 
$$\|A\|= \sup \{ \| Ax\| : \|x\|=1\}$$
and it is denoted by $\rlin \h$. Every real linear operator $A$ can be decomposed as 
\begin{equation}\label{complexantirep}
A=A_0+A_1
\end{equation} 
with $A_0=\frac 12 (A-iAi)$ complex linear and $A_1=\frac 12(A+iAi)$ antilinear. 

A number $\lambda \in \Complex$ is in the spectrum $\sigma(A)$ of $A$ if $A-\lambda$ is not invertible in $\rlin \h$. The number $\lambda$ is an eigenvalue and is in the point spectrum $\sigma_p(A)$ if there exists a non-zero vector $x\in \h$ such that $(A-\lambda)x=0$. The number $\lambda$ is in the approximate point spectrum $\sigma_a(A)$ if there is a sequence of unit vectors $\{x_n\}_{n\geq 1} \subset \h$ such that $(A-\lambda)x_n \xrightarrow[n\to\infty]{} 0$ in $\h$. 
The number $\lambda$ is in the compression spectrum $\sigma_c(A)$ if the range $\R (A-\lambda)$ is not dense in $\h$. 

The adjoint $A^*$ of a real linear operator $A$ is defined by 
$$\re (Ax,y) = \re (x, A^* y) \quad \text{ for all $x,y\in \h$,} $$
where $(\cdot,\cdot)$ is the inner product in $\h$. Equivalently, using the representation \eqref{complexantirep}, we can define 
$$A^*=A_0^*+A_1^*, $$
where $A_0^*$ and $A_1^*$ satisfy 
$ (A_0x,y)= (x, A_0^*y)$ and  $(A_1 x,y) = \overline{(x, A_1^*y)} $
for all $x,y\in\h$, respectively. A real linear operator $A$ is said to be self-adjoint if $A=A^*$. It is unitary if it is bijective and an isometry, i.e. $\| A x \| = \|x\|$ for all $x\in \h$. 
A unitary operator is called complex unitary or antiunitary if it is complex linear or antilinear, respectively. 

The spectral theory for real linear operators is not fully understood. It is known that the spectrum of a real linear operator is compact. It is possible for the spectrum to be empty. However, the spectrum of a self-adjoint real linear operator is known to be non-empty. It is not necessarily real but it is  symmetric with respect to the real line. The spectrum of an antilinear operator is always circularly symmetric with respect to the origin. See \cite{reallinearoperator} for more details. 

Among the simplest antilinear operators are the so-called conjugations. 
\begin{defn}
An antilinear operator $\kappa \in \rlin \h$ is a conjugation on $\h$ if it is an involution, i.e. it satisfies $\kappa^2 = I$. 
\end{defn}
For a conjugation, we have the following. 
\begin{prop}\label{pointspectrumofconjugation}
For a conjugation $\kappa$ on $\h$ there holds 
$$\sigma_p(\kappa) = \sigma(\kappa) = \{ e^{i\theta} : \theta\in\Real \} .$$  
\end{prop}
\begin{proof}
The conjugation $\kappa$ being antilinear, its spectrum is circularly symmetric with respect to the origin. For any nonzero $x\in \h$ we have $(\kappa -1)(\kappa + 1)x = (\kappa^2 -1)x = 0$. 
Now, either $(\kappa+1)x=0$, whence $(\kappa -1) ix = 0$, or $(\kappa+1)x\neq 0$, whence $(\kappa-1)y=0$ with $y=(\kappa+1)x$. Thus $1\in\sigma_p(\kappa)$. 

On the other hand, if $r\neq 1$ is a non-negative real number, then $(\kappa+r)(\kappa-r)x_n=(1-r^2)x_n$ does not tend to zero for any sequence of unit vectors $\{x_n\}$. That is, $r$ is not in the approximate point spectrum $\sigma_a(\kappa)$. Similarly, $(\kappa^*-r)(\kappa^*+r)x=(1-r^2)x\neq 0$ for any $x\neq 0$. Thus $r$ is not in the compression spectrum $\sigma_c(\kappa)$. 
\end{proof}

Unitary conjugations being norm-preserving are of natural interest. Moreover, for antilinear (as well as for complex linear) operators, being involutory, self-adjoint or unitary are properties of which any two imply the third.  

Given an orthonormal basis $\{e_n\}$ of $\h$, define $\kappa$, the conjugation with respect to $\{e_n\}$, to be the antilinear operator for which $\kappa e_n = e_n$ for all $n$. Clearly, $\kappa$ is unitary. The content of the next proposition is the converse, i.e. for any unitary conjugation $\kappa$ there is an orthonormal basis of $\h$ with respect to which $\kappa$ can be defined. 

\begin{prop}\label{eigbasisforconj}
Let $\kappa$ be a unitary conjugation on $\h$. 
Then there is an orthonormal basis $\{e_n\}_{n=1}^\infty$ of $\h$ such that $\kappa e_n = e_n$. 
\end{prop}
\begin{proof}
By Proposition \ref{pointspectrumofconjugation} there is a normalized eigenvector $e_1$ of $\kappa$ such that $\kappa e_1 = e_1$. 
Now take a vector $y\in \{e_1\}^\perp$. 
Then also $\kappa y \in \{e_1\}^\perp$ since $(\kappa y , e_1 ) = (\kappa e_1 , y) = (e_1,y) = 0$, i.e. $\spn \{e_1\}$ is a reducing subspace for $\kappa$. 
Let $P_1$ be the orthogonal projection onto $\spn \{e_1\}$. 
Then $P_1^\perp \kappa P_1^\perp$ is a conjugation on $P_1^\perp \h$ and has a unit eigenvector $e_2 \in \{e_1\}^\perp$. 
Continuing by induction, let $P_n$ be the orthogonal projection onto $\spn \{ e_1,\ldots,e_n\}$ and take $e_{n+1}$ to be the unit eigenvector of $P_n^\perp \kappa P_n^\perp$ in $(\spn \{ e_1, \ldots, e_n\})^\perp$. 
Then the set $\{e_n\}_{n=1}^\infty$ forms the required orthonormal basis of eigenvectors of $\kappa$. 
\end{proof}

Although unitary conjugations are simply defined and have nice properties, they are not necessarily simple to operate with as illustrated by the following example. 

\begin{ex}
Let $S$ is the Beurling transform on $L^2(\Complex)$ defined as a principal value integral 
$$ Sf(z) = -\frac{1}{\pi} \lim_{\eps\to 0} \int_{|z-w|>\eps} \frac{f(w)}{(z-w)^2}\, dw_1dw_2, \quad w=w_1+iw_2, $$
and $\tau$ the complex conjugation $f\mapsto \overline f$ on $L^2(\Complex)$. Then the operator $S\tau$ is a unitary conjugation on $L^2(\Complex)$. Namely, $S$ is complex unitary and it holds $S^{-1} = \tau S \tau$. Then $(S\tau)(S\tau) = I$ and $(S\tau)^* = \tau^* S^* = \tau \tau S \tau = S\tau$. For more details on the Beurling transform and their applications, see for example \cite{AstalaEtal:Elliptic}. 
\end{ex}

Classically, an orthogonal projection is an operator that is the identity on a closed subspace and zero on the orthogonal complement of this subspace. Analogously, we define an 'antilinear orthogonal projection' to be an antilinear counterpart for orthogonal projections in the following sense.

\begin{defn}
Let $\M$ be a closed subspace of $\h$. An operator $F$ on $\h$ is said to be an antilinear orthogonal projection if its restriction to $\M$, $F|_{\M}$, is a unitary conjugation on $\M$ and $F$ is zero on $\M^\perp$. \footnote{$F$ is called a partial conjugation in \cite{garcia2}.} 
\end{defn}
Note that if $P$ is the orthogonal projection onto $\M$ and $\tau$ is some unitary conjugation on $\h$, $F$ is not given in general as $P\tau P$. 

Orthogonal projections are in one-to-one correspondence with closed subspaces of $\h$. In strong contrast to this, there is a multitude of antilinear projections for any given closed subspace $\M$ of $\h$. 

\begin{ex}
Let $\{e_n\}$ be a set of orthonormal vectors spanning a closed subspace $\M\subset\h$. Then also the set $\{e^{i\theta_n}e_n\}$, $\theta_n\in\Real$, is orthonormal and spans $\M$. The operators $F_1$ and $F_2$ defined by $F_1x = \sum_n (e_n, x) e_n$ and $F_2x = \sum_n (e^{i\theta_n}e_n, x) e^{i\theta_n}e_n$ for all $x\in\h$ are both antilinear projections onto $\M$. However, clearly $F_1\neq F_2$. 
\end{ex}

The basis dependence of antilinear projections, or unitary conjugations when $\M=\h$, might seem unappealing operator theoretically. However, it allows for defining a natural basis in the sense of Proposition \ref{eigbasisforconj}. On the other hand, there are instances where antilinearity is the key to basis independence as is illustrated by the following example. 

\begin{ex}
In so-called bipartite quantum systems, a state $\sigma\in\h \otimes \h$ can be represented as 
$$ \sigma = \sum_n v_n \otimes e_n \quad\text{with $\sum_n \|v_n\|^2<\infty$,} $$
where $\{e_n\}$ is an orthonormal basis of $\h$. Defining $L_\sigma$ as the unique antilinear (not complex linear) operator such that $L_\sigma e_n = v_n$ leads to the representation $\sigma = \sum_n L_\sigma e_n \otimes e_n$. However, this is independent of the choice of the orthonormal basis. See \cite{AV:EPR} for a detailed account. It is crucial that $L_\sigma$, the so-called relative state operator, be antilinear for this representation to be basis independent. This antilinear representation for states is advantageous when discussing quantum entanglement, for example in the study of Einstein--Podolsky--Rosen states \cite{AV:EPR} and of quantum teleportation \cite{KKAJQTelep}.
\end{ex}

Using the existence of a natural basis for a unitary conjugation in the sense of Proposition \ref{eigbasisforconj}, any two unitary conjugations are related in a simple way. 

\begin{prop}
Let $\tau$ and $\kappa$ be unitary conjugations on $\h$. Then there is a complex linear unitary operator $U$ such that $\tau = U^*\kappa U$. 
\end{prop}
\begin{proof}
Let $\{e_n\}$ (resp. $\{f_n\}$) be the orthonormal basis of $\h$ for which $\tau e_n = e_n$ (resp. $\kappa f_n = f_n$) for all $n\geq 1$. Define $U$ to be the complex linear operator such that $Ue_n = f_n$. Then $U$ is unitary. Moreover
$ \kappa U x = \kappa U \sum_n a_n e_n = \sum_n \overline{a_n} f_n $ 
and 
$ U\tau x = U\tau\sum_n a_n e_n = \sum_n \overline{a_n} f_n $
for all $x\in\h$. Thus $\tau = U^*\kappa U$.   
\end{proof}

Clearly, if $\tau$ and $\kappa$ are unitary conjugations, then $\kappa\tau$ is complex linear and unitary. Godi\u{c} and Lucenko have proved that the converse holds, i.e., if $U$ is a complex linear unitary operator, then there are two conjugations $\tau$ and $\kappa$ such that $U=\tau\kappa$ \cite{GL:RUO2I}. 

\section{Antilinear self-adjoint operator as a spectral integral}

In this section we develop how to represent an antilinear self-adjoint operator $A$ on $\h$ in the form 
\begin{equation}\label{intrep}
A=\int_{\sigma(A)\cap\Real_+} \lambda \,dF(\lambda) . 
\end{equation}
For comparison, recall that using the spectral resolution, a complex linear self-adjoint operator $H$ can be written as a spectral integral 
$$H=\int_{\sigma(H)} \lambda\, dE(\lambda) $$
where $E$ is a spectral measure on $\sigma(H)$. The spectral measure is defined on the $\sigma$-algebra of Borel subsets of $\sigma(H)$ and its values are orthogonal projections on $\h$. In addition, the spectral measure is required to be such that $E(\sigma(H))=I$ and $E(\bigcup_n M_n)=\sum_n E(M_n)$ whenever $\{M_n\}$ is a disjoint sequence of sets. As an analogue $F$ in \eqref{intrep} is an antilinear spectral measure to be defined below. 

To this end, let us start with the polar decomposition of $A$. Our exposition follows that of \cite[p. 1346]{herbut:basicalgebra} where general antilinear operators on finite dimensional spaces are considered. 
Recall that a self-adjoint complex linear operator $B$ on $\h$ is positive if $(Bx,x)\geq 0$ for all $x\in\h$. 

\begin{prop}\label{polardecomp}
Every self-adjoint antilinear operator $A$  on $\h$ can be written in the polar form $A=|A| \tau = \tau |A|$, where $\tau$ is a unitary conjugation and $|A|$ is defined to be the complex linear positive square root of $A^*A$. 
\end{prop}
\begin{proof}
Define $|A|$ to be the unique positive complex linear square root of the operator $A^*A$. 
We have for all $x\in\h$ that $\||A|x\| = \|Ax\|$. 
It follows that the null spaces of $|A|$ and $A$ coincide, $\N(A) = \N(|A|)$. Since $A$ and $|A|$ are self-adjoint, it holds $\h=\N(A)\oplus\overline{\R(A)}=\N(|A|)\oplus\overline{\R(|A|)}$ and thus the closures of the ranges of $A$ and $|A|$ also coincide.  Denote $\V_1=\overline{\R(|A|)}=\overline{\R(A)}$ and $\V_2=\N(|A|)=\N(A)$. 
The isometricity of $A$ and $|A|$ implies that there is a unique anti-linear isomorphism $U_1$ on $\V_1$ such that $A=U_1|A|$. Namely, for every $x\in\V_1=\overline{\R(|A|)}$, there is $y\in\h$ such that $x=|A|y$. Define then $U_1x = Ay$. (Note that since $\N(|A|)=\N(A)$, for the inverse images we have $A^{-1}(x)=|A|^{-1}(x)$, and the definition of $U_1$ does not depend on choice of the preimage $y$. ) By the isometricity of $A$ and $|A|$, $U_1$ is unitary. 

Take then an arbitrary antilinear self-adjoint isomorphism $U_2$ on $\V_2$. This can be done by choosing an orthonormal basis for $\V_2$ and defining $U_2$ to be the conjugation with respect to that basis. 
Finally, define $\tau$ on $\h$ by $Ux=U_1x_1 + U_2x_2$, where $x_1\in\V_1$ and $x_2\in\V_2$. 

Setting $H_2=\tau |A| \tau^*$ we have the factorization $A=H_2\tau$. However, as $A$ and $H_2$ are self-adjoint, $H_2^2=H_2H_2^*=AA^* = A^*A$. Hence by the uniqueness of the square root, we have $H_2=|A|$. Thus $A=\tau |A|=|A|\tau $. 

Moreover, $A=\tau |A|=\tau^* |A|$ so that $\tau y=\tau^*y$ for all $y\in\V_1$. As $U_2$ was already chosen to be self-adjoint on $\V_2$, the anti-unitary $\tau = U_1 \oplus U_2$ is also self-adjoint, thus a unitary conjugation. 
\end{proof}

Remark that in the proof above the self-adjointness of $A$ was needed only in proving that $\tau$ is self-adjoint on $\V_1$. Otherwise, the assumption $AA^*=A^*A$ would suffice. Note also that $|A|$ is unique and $\tau$ is non-unique only on $\N(A)$. 

All the necessary information about the spectrum of a self-adjoint antilinear operator $A$ is given by $|A|$. Indeed, $\sigma(|A|)$ lies on the non-negative real line $\Real_+$. We have that $r\in\sigma(|A|)$ if and only if $r^2 \in \sigma(|A|^2) = \sigma(A^2)$ and that the latter is equivalent with $r \in \sigma(A)$ by \cite[Proposition 2.15]{reallinearoperator}. In addition, it is known that the spectrum of an antilinear operator is circularly symmetric with respect to the origin. The fact that the spectra of $A$ and $|A|$ are in this manner closely related leads to the following definition. 

\begin{defn}
Define the antilinear spectral measure $F$ for an antilinear self-adjoint operator $A$ on $\h$ by 
\begin{equation*}
F(M)=E(M)\tau 
\end{equation*}
for every Borel subset $M$ of $\sigma(A)\cap \{\lambda\geq 0\}=\sigma(|A|)$. Here $E$ is the spectral measure for $|A|=(A^*A)^{1/2}$ and $\tau$ is the unitary conjugation given by Proposition \ref{polardecomp}. Denote for further convenience by $\Sigma$ the $\sigma$-algebra of Borel subsets of $\sigma(A) \cap \{\lambda\geq 0\} = \sigma(|A|)$. 
\end{defn}

For $F$ to be appropriate for its role, it is crucial that $E(M)\tau = \tau E(M)$ for all $M\in\Sigma$. Even though $\tau$ is antilinear, standard textbook methods (like in \cite[Theorem 10.2]{conway:oper} or in \cite[Theorem 40.2]{halmos:hilbert}) can be used to prove this. It should be observed, however, that properties used in proving the following lemma require that the complex linear operator $H$ be self-adjoint, not normal. 

\begin{lemma}\label{commutate}
Let $\tau$ be antilinear and $H=\int \lambda\, dE(\lambda)$ be complex linear and self-adjoint where $E$ is the spectral measure for $H$. If $\tau H = H \tau$, then $E(M)\tau = \tau E(M)$ for all $M\in\Sigma$.
\end{lemma}
\begin{proof}
We have $p(H)\tau = \tau p(H)$ for every real polynomial $p$. Then for all $x,y\in\h$
\begin{equation}
\int p(\lambda)\, d(E(\lambda)x, \tau^* y) = (p(A)x,\tau^* y) = \overline{(p(A)\tau x, y)} = \overline{\int p(\lambda)\, d(E(\lambda) \tau x , y)}, 
\end{equation}
from which 
we can infer that
\begin{equation}
(\tau E(M) x , y) = \overline{(E(M) x , \tau^* y)} = (E(M) \tau x,y ). 
\end{equation}
Thus $E(M)$ and $\tau$ commute. 
\end{proof}

From this it follows that an antilinear self-adjoint operator may be represented in the form \eqref{intrep}. Analogously to classical spectral measures, the antilinear spectral measure satisfies the following: 
\begin{enumerate}[(i)]
\item The values of $F$ are antilinear projections, i.e. 
$$ F(M)^2=E(M) \ \text{ and }\ F(M)^* =(E(M)\tau)^* = E(M)\tau = F(M) .$$
\item It holds $F(\sigma(A)) = \tau$ and
\item $F(\bigcup_n M_n ) = \sum_n F(M_n)$ for any disjoint sequence of sets $\{M_n\}$. 
\end{enumerate}

\section{Weyl--von Neumann theorem for antilinear self-adjoint operators}

In this section we aim to prove an analogue of the Weyl--von Neumann theorem. Naturally, the question arises why the polar decomposition of Proposition \ref{polardecomp} with the classical Weyl--von Neumann theorem would not provide the wanted result directly. By this we mean that we certainly may write $A=|A|\tau$ and use the representation $|A|=D+K$ of the Weyl--von Neumann theorem. This provides us with a compact operator $K$ and an operator $D$ diagonal with respect to an orthonormal basis $\{e_n\}$ of $\h$. Then clearly $A=D\tau + K\tau$ where $K\tau$ is compact. However, we cannot claim that $D\tau$ is diagonal. This would be the case only if the orthonormal basis diagonalizing $\tau$ would happen to be the same as the one diagonalizing $D$. 

\subsection{The main theorem}
Recall the notion of diagonalizability. It is in effect unitary diagonalizability, and as such more stringent than other, more general definitions of diagonalizability for Hilbert space operators, cf. \cite{Herrero:triangular}.  

\begin{defn}
An operator $A$ on $\h$ is diagonalizable if there exists an orthonormal basis $\{e_n\}$ of $\h$ such that $Ae_n=a_ne_n$ for all $n$ for some complex numbers $a_n$, i.e. if there is an orthonormal set of eigenvectors spanning $\h$. Then $A$ is said to be diagonal with respect to $\{e_n\}$.  
\end{defn}

Let us glance at the finite dimensional diagonalizability first. It is interesting in its own right due to its frequent occurrence in applications. In addition, the proof of the main theorem relies on it. 
\begin{prop}\label{usetakagi}
Let $\h$ be a finite dimensional Hilbert space and $A\in\B(\h)$ an antilinear self-adjoint operator on $\h$. Then $A$ is unitarily diagonalizable. 
\end{prop}
\begin{proof}
We can factor $A=A_\#\tau$ where $\tau$ is represented by complex conjugation on $\Complex^n$. Then $(A_\#\tau)^*=\tau^*A_\#^* = A_\#^T \tau$, that is $A_\#$ is complex symmetric. By the Takagi factorization (cf. e.g. \cite{hornjohnson1}) there is a unitary matrix $U$ such that $UA_\#U^T = D$ where $D$ is diagonal. Thus $D\tau = UAU^T\tau = U D\tau U^*$, i.e. $A$ is diagonalizable since $D\tau$ is diagonal with respect to the standard basis.  
\end{proof}

Diagonalizability in the antilinear case is not a trivial matter, though. Recall that almost all complex matrices are diagonalizable in the sense that the probability is one for a randomly picked complex linear operator on a finite dimensional Hilbert space to be diagonalizable.  However, diagonalizable antilinear operators are a lot more scarce. The probability of a randomly picked antilinear operator on an $n$-dimensional Hilbert space to be diagonalizable is $2^{-n(n-1)/2}$, see \cite{HuhtanenPeramaki} for details.

Analogously to the complex linear case, we define singular values and Schatten $p$-class operators as follows. 
\begin{defn}
Define the singular values $s_n(A)$, $n=1,2,\ldots$, of a compact antilinear operator $A$ as the eigenvalues of the complex linear positive operator $|A| = (A^*A)^{1/2}$ in non-increasing order of magnitude. We say that $A$ is in Schatten $p$-class $\B_p(\h)$, $1\leq p < \infty$, if 
$$ \|A\|_p = \left( \sum_n s_n(A)^p \right)^{1/p} < \infty. $$ 
\end{defn}

Note that if we factor a compact antilinear operator $A$ as $A = A_\# \tau$ with $\tau$ a unitary conjugation and $A_\#=A\tau$, then $s_n(A) = s_n(A_\#)$. This follows from the fact $\sigma(|A_\#|) = \sigma(|A|)$ which holds since $|A_\#| = (\tau^*A^*A\tau)^{1/2}$ and $\sigma(|A_\#|)$ is real. 

From this connection between $A$ and $A_\#$, the following is immediate. The rank of an antilinear operator is defined to be the dimension of its range. 

\begin{lemma}\label{lem:schatten}
For an antilinear operator $A$ of rank at most $n$, there holds $\|A\|_p \leq n^{1/p} \|A\|$. 
\end{lemma}

Now we can state the analogue of the Weyl--von Neumann theorem. 

\begin{thm}\label{WeylvonN}
Let $A$ be a self-adjoint antilinear operator on $\h$, $\eps >0$ and $1<p<\infty$. Then there is a diagonalizable self-adjoint antilinear operator $D$ such that $A-D$ is compact and $\|A-D\|_p<\eps$. 
\end{thm}

The steps in proving that $A$ is the sum of a diagonalizable operator $D$ and a Schatten $p$-class operator follow those taken in \cite[pp. 212--213]{conway:oper}; see also \cite[Chapter X]{katopert}, and \cite{kuroda:weylvonn} for extension from $p=2$ to $1<p<\infty$. A modification is needed: in the following proposition we have to restrict that $f=\tau f\in\h$. 

\begin{lemma}\label{reduce}
Let $A=|A|\tau$ be a self-adjoint antilinear operator and $\tau f=f \in\h$. Then for any $\eps>0$ there is a finite rank projection $P$ and a self-adjoint antilinear operator $K\in\B_p(\h)$, $1<p<\infty$, such that $f\in P\h$, $\|K\|_p < \eps$, and $A+K$ is reduced by $P$. In addition, $P\tau = \tau P$. 
\end{lemma}
\begin{proof}
Factor $A$ in the form $A=\tau |A| = |A|\tau$ as in Lemma \ref{polardecomp}. 
The self-adjoint operator $|A|$ has a spectral decomposition $|A|=\int \lambda\, dE(\lambda)$ with respect to the spectral measure $E$.
Assume $\sigma(|A|)\subset [a,b]$, where $[a,b]$ is an interval in the non-negative real line. 
Partition $[a,b]$ into $n$ equal subintervals $M_1,\ldots,M_n$ each of length $\frac{b-a}{n}$ and let $\lambda_k$ be the midpoint of the interval $M_k$. 
Set $f_k = E(M_k)\tau f = E(M_k)f$ and $g_k=f_k/\|f_k\|$ if $f_k\neq 0$ or $g_k=0$ otherwise. 
Denote for convenience $c_k = 1/\|f_k\|$. 
Thus we have that $g_k\in E(M_k)\tau \h$, whence $g_j\perp g_k$ for $j\neq k$. 
Then 
\begin{equation*}
\|(A-\lambda)g_k\| = \| |A|\tau c_k E(M_k) f - \lambda c_k E(M_k) f \| = \| (|A| - \lambda ) g_k \| 
 \leq  \frac{b-a}{n}. 
\end{equation*}
Denoting by $P$ the orthogonal projection onto $\spn \{g_k\}_{k=1}^n = \spn \{f_k\}_{k=1}^n$, we get 
\begin{equation}
\|P^\perp Ag_k \| = \| P^\perp (A-\lambda) g_k \| \leq \frac{b-a}{n} .
\end{equation}
It holds $Ag_k \in AE(M_k)\h = E(M_k)A\h \subset E(M_k)\h$ so that $Ag_k \perp g_j$ for $k\neq j$. Hence
$$P^\perp A g_k = A g_k - \sum_j (Ag_k,g_j) g_j = Ag_k - (Ag_k,g_k) g_k \in E(M_k)\h $$
so that also $P^\perp Ag_k \perp P^\perp Ag_j$ for $k\neq j$. 
Using this orthogonality, we have 
\begin{eqnarray*}
\|P^\perp A P h\|^2 &=& \left\| \sum_k (g_k, h) P^\perp A g_k \right\|^2 = \sum_k |(g_k, h)|^2 \|P^\perp A g_k \|^2 \\
&\leq & \|h\|^2 \left( \frac{b-a}{n} \right)^2
\end{eqnarray*}
for all $h\in\h$. Thus $\|P^\perp A P\| \leq (b-a)/n$ with $P^\perp A P$ having rank at most $n$. By Lemma \ref{lem:schatten}, $\|P^\perp A P\|_p \leq (b-a) n^{-1/q}$ with $\frac 1p + \frac 1q = 1$. 

Define $B=PAP + P^\perp AP^\perp$ and $K=-P^\perp AP - P AP^\perp$. Then $B$ and $K$ are self-adjoint antilinear operators and $A=B-K$. The operator $B$ is reduced by $P$, $K$ has finite rank, and $\| K \|_p\leq 2 (b-a)/n^{1/q}$ which can be made arbitrarily small with a suitable choice of $n$. 

Finally,
\begin{eqnarray*}
\tau P x &=& \tau \sum_k (x, c_kE(M_k)f ) c_kE(M_k)f =  \sum_k (c_kE(M_k)f , x) c_kE(M_k) \tau f \\
&=& \sum_k (\tau x, \tau c_kE(M_k)f ) c_kE(M_k) f = P\tau x
\end{eqnarray*}
for all $x\in\h$ so that $P\tau = \tau P$. 
\end{proof}

Using the previous lemma, the Weyl--von Neumann theorem for antilinear self-adjoint operators can be proven. 

\begin{proof}[Proof of Theorem \ref{WeylvonN}]
Let $\{e_n\}$ be an orthonormal basis of $\h$ such that $\tau e_n = e_n$ for all $n$. Apply the preceding Lemma \ref{reduce} with $f=e_1$ to get a finite rank projection $P_1$ and a self-adjoint operator $K_1\in\B_p(\h)$ with $\|K_1\|_p<\eps/2$ such that $A+K_1$ is reduced by $P_1$ and $e_1\in P_1\h$. 
Apply the lemma again to $(A+K_1)|_{(P_1^\h)^\perp}$ with the vector $f = P_1^\perp e_2 = \tau f$ to get a self-adjoint operator $K_2\in\B(P_1^\perp\h)$ and a projection $P_2$ such that $P_1^\perp e_2 \in P_2\h$, $\|K_2\|_p<\eps/2^2$ and $A+K_1+K_2$ is reduced by $P_2$. 
Extend $K_2$ to all of $\h$ by $K_2y=0$ for all $y\in P_1\h$. Note that $e_1,e_2\in (P_1+P_2)\h$. 

By induction we get a sequence of finite rank projections $\{P_n\}$ and a sequence of self-adjoint operators $\{K_n\}$ such that 
\begin{enumerate}
\item $\|K_n\|_p<\eps/2^n$
\item $P_jP_k=0$ for $j\neq k$
\item $e_n\in (\sum_1^n P_k)\h$
\item $A+K_1+\cdots +K_n$ is reduced by $(P_1+\cdots+ P_n)$
\item $K_n(P_1+\cdots +P_{n-1})=0$. 
\end{enumerate}

Set $K=\sum_n K_n$, $D=A+K$ and $D_n=D|_{P_n\h}$. Then $\|K\|_p<\eps$ by property (i) and $D$ is self-adjoint. Properties (ii) and (iii) imply that $\sum_n P_n=I$. Properties (iv) and (v) imply that $D$ is reduced by $P_n\h$ for all $n$ and $D=\bigoplus_n D_n$. 

Since each of the spaces $P_n\h$ is finite dimensional, by Proposition \ref{usetakagi} there is an orthonormal basis of $P_n\h$ that diagonalizes $D_n$. Thus $D$ is a diagonalizable operator.  
\end{proof}

Note that the diagonalizable antilinear operator can be assumed to have a non-negative diagonal. Namely, assume $D$ on $\h$ is diagonal with respect to the orthonormal basis $\{e_n\}$, i.e. $D e_n = d_ne_n = |d_n| e^{i\theta_n}$. Then for all $x = \sum_n(x,e_n)e_n \in\h$
$$ Dx = \sum_n ( e_n, x) d_ne_n = \sum_n ( e^{i\theta_n/2} e_n , x) |d_n| e^{i\theta_n/2} e_n $$
so that $D$ is diagonal with respect to the orthonormal basis $\{e^{i\theta_n/2} e_n \}$ with  non-negative diagonal.

\subsection{Finite rank generalization to real linear operators}

The question whether there is a more general version of Theorem \ref{WeylvonN} is treated in finite dimensions. 

\begin{lemma}\label{commoneigvect}
Let $N$ be a complex linear normal operator and $S$ an antilinear self-adjoint operator on a finite dimensional Hilbert space $\h$ such that $N S = S N^*$. Then $N$, $N^*$ and $S$ have a common eigenvector. 
\end{lemma}
\begin{proof}
Let $y\in \W = \{x\in\h : Nx = \lambda x\} = \{x\in\h : N^*x=\overline \lambda x \}$ for some eigenvalue $\lambda\in\Complex$ of $N$. Then $N S y = S N^* y = S \overline\lambda y = \lambda S y$. Thus $ S y \in \W$. Therefore, the subspace $W$ is $S$-invariant. 

The restriction $S|_\W$ is an antilinear self-adjoint operator on $\W$, and as such \cite{reallinearoperator} has an eigenvalue $r\in\Real$ corresponding to an eigenvector $z\in\W$. Then $z$ is the desired common eigenvector. 
\end{proof}

\begin{prop}
Let $A=N+S$ be a real linear operator on a finite dimensional Hilbert space $\h$, where $N$ is complex linear normal, $S$ is antilinear self-adjoint and they satisfy $N S = S N^*$. Then $A$ is unitarily diagonalizable. 
\end{prop}
\begin{proof}
By Lemma \ref{commoneigvect} there exists a unit vector $e_1\in\h$ such that $Ne_1=\lambda e_1$, $N^*e_1=\overline\lambda e_1$ and $S e_1 = r e_1$ for some $\lambda\in\Complex$ and $r\in\Real$. Split $\h$ as $\h=\spn\{e_1\} \oplus \spn\{e_1\}^\perp$. Then obviously $Ax\in\spn\{e_1\}$ for $x\in\spn\{e_1\}$. But we have also 
$$ (e_1, Ay) = (e_1, Ny) + (e_1, S y) = (N^*e_1, y) + ( S e_1, y) = \lambda (e_1,y) + r(e_1,y)=0 $$
for every $y\in\spn\{e_1\}^\perp$. Thus $Ay \in \spn\{e_1\}^\perp$. Hence, $A=D_1\oplus A_1$, where $D_1$ is trivially diagonal on $\spn\{e_1\}$ and $A_1$ is a real linear operator on $\spn\{e_1\}^\perp$ satisfying the assumptions of the proposition. Iterate the previous to finally get an orthonormal basis $\{e_n\}$ of $\h$. 
\end{proof}

This raises the question if, similarly as in the proof of Theorem \ref{WeylvonN}, a reduction to the finite dimensional case can be made to prove an analogue of the Weyl--von Neumann Theorem for real linear operators satisfying the commutation property above. However, at this point, it remains unclear whether this can be done. 

\subsection{Complex symmetric operators}

In what follows, we make some remarks on complex symmetric operators and their relation to antilinear self-adjoint operators. A complex linear operator $S$ on $\h$ is called $\tau$-symmetric if 
$$\tau S^* \tau = S, $$
where $\tau$ is a unitary conjugation.  It is called complex symmetric in general if it is $\tau$-symmetric with respect to some unitary conjugation $\tau$. Complex symmetric operators have been the object of recent investigations, and they have been shown to include a variety of important operators, for example all normal operators, Hankel operators, compressed Toeplitz operators and many standard integral operators \cite{garcia,garcia2,garcia:newclasses}. 

Complex symmetric operators and antilinear operators are related in the following manner. 

\begin{prop}
\begin{enumerate}
\item If $S$ is $\tau$-symmetric on $\h$, then the operator $S\tau$ is antilinear self-adjoint. 
\item If the operator $A$ on $\h$ is antilinear self-adjoint, then $A\tau$ is $\tau$-symmetric for any unitary conjugation $\tau$.  
\end{enumerate}
\end{prop}
\begin{proof}
The first assertion follows from 
$$ (S\tau)^* = \tau^*S^* = \tau S^* \tau^2 = S\tau. $$
The second follows from 
$$ (A\tau)^* = \tau^* A^* = \tau (A\tau) \tau . \qedhere $$
\end{proof}

This connection may be useful in some contexts as definineg antilinear self-adjointness is basis independent whereas complex symmetricity is defined essentially through the choice of an orthonormal basis by fixing a unitary conjugation. 

In finite dimensions, by the Takagi factorization, a matrix $S$ is complex symmetric, $S=S^T$, if and only if it is unitarily condiagonalizable, i.e. there exists a unitary matrix $U$ and a diagonal matrix $D$ (with non-negative entries) such that $S=UDU^T$. Using Theorem \ref{WeylvonN} and the correspondence with antilinear operators, it can be seen that, in the infinite dimensional case, complex symmetric operators are arbitrarily close to condiagonalizable operators. Here, the analogue of the transpose of $U$ in the infinite dimensional setting is $\tau U^* \tau$. 

\begin{prop}
Let $S$ be a complex symmetric operator on $\h$ with respect to the unitary conjugation $\tau$. Then for any $\eps>0$ there exists a complex linear unitary operator $U$ and a diagonalizable complex linear operator $D$ such that 
$$\| S - UD\, \tau U^*\tau \| < \eps .$$
\end{prop}
\begin{proof}
Let $\{e_n\}$ be the orthonormal basis of $\h$ for which $\tau e_n = e_n$. Since $S\tau$ is antilinear and self-adjoint, by Theorem \ref{WeylvonN}, there is an antilinear diagonalizable operator $\tilde D$ such that $\|S\tau-\tilde D\|\leq \|S\tau -\tilde D\|_p < \eps$. Let $\{f_n\}$ be the orthonormal basis diagonalizing $\tilde D$, i.e. $\tilde Df_n=d_nf_n$ where $d_n \geq 0$. Define $U$ by $Ue_n = f_n$. Then we have for all $x\in\h$
\begin{eqnarray*}
S\tau x - Dx &=& S\tau x - \sum_n d_n (f_n, x) f_n = S\tau x - \sum_n d_n (Ue_n, x) Ue_n \\
&=& S\tau x - U \sum_n d_n (e_n , U^* x ) e_n = (S - U D \, \tau U^* \tau) \tau x,  
\end{eqnarray*}
where $D$ is the complex linear diagonal operator with respect to $\{e_n\}$. From this we can infer, upon using the norm estimate given  by Theorem \ref{WeylvonN}, that
$ \| S - U D \, \tau U^* \tau \| < \eps$. 
\end{proof}

We also have the following. 

\begin{prop}
Let $S$ be a complex symmetric operator on $\h$ with respect to the unitary conjugation $\tau$. Then there exist a unitary conjugation $\kappa$ and a diagonalizable complex linear operator $D$ with non-negative diagonal such that $\kappa$ and $D$ are diagonalized with respect to the same orthonormal basis and 
$$ \| S - \tau D \kappa \|_p < \eps .$$
\end{prop}
\begin{proof}
Since $\tau S$ is antilinear self-adjoint, by Theorem \ref{WeylvonN} there exists an antilinear operator $\tilde D$, diagonal with respect to an orthonormal basis $\{e_n\}$, with a non-negative diagonal, such that $\| \tau S - \tilde D \|_p < \eps$. We can factor $\tilde D =D \kappa$ where $D$ is complex linear diagonal with respect to $\{e_n\}$ and $\kappa$ is the unitary conjugation with respect to $\{e_n\}$. Then we have $\| S - \tau D \kappa \|_p < \eps$. 
\end{proof}

\bibliographystyle{plain}
\bibliography{references_santtu}
\end{document}